\documentclass{basm}

\usepackage{amsmath}

\theoremstyle{plain}
\newtheorem{theorem}{Theorem}[section]

\newtheorem{corollary}[theorem]{Corollary}

\newtheorem{lemma}[theorem]{Lemma}

\newtheorem{proposition}[theorem]{Proposition}

\theoremstyle{definition}
\newtheorem{problem}[theorem]{Problem}

\newtheorem{algorithm}[theorem]{Algorithm}

\def\img#1{\mathrm{Im}(#1)}
\def\aut#1{\mathrm{Aut}(#1)}
\def\Z{\mathbb Z}

\def\Q{\mathcal Q}

 \author{David Stanovsk\'y, Petr Vojt\v echovsk\'y}

\title{Central and medial quasigroups of small order}

\RunningHead{David Stanovsk\'y, Petr Vojt\v echovsk\'y}%
            {Central and medial quasigroups of small order}



\Abstract{We enumerate central and medial quasigroups of order
less than $128$ up to isomorphism, with the exception of those
quasigroups that are isotopic to $C_4\times C_2^4$, $C_2^6$,
$C_3^4$ or $C_5^3$. We give an explicit formula for the number of
quasigroups that are affine over a finite cyclic group. }

\CopyRight { David Stanovsk\'y, Petr Vojt\v echovsk\'y, 2016 \\
Research partially supported by the GA\v CR grant 13-01832S
(Stanovsk\'y), the Simons Foundation Collaboration Grant
210176 (Vojt\v{e}chovsk\'y) and the University of Denver PROF grant (Vojt\v{e}chovsk\'y).}

\Address
{{\sc David Stanovsk\'y}\\
{Department of Algebra\\
Charles University\\
Sokolovsk\'a 83\\
186 75 Praha 8, CZECH REPUBLIC}\\
E-mail: \emph{stanovsk@karlin.mff.cuni.cz}
\medskip

{\sc Petr Vojt\v echovsk\'y}\\
{Department of Mathematics\\
University of Denver\\
2280 S Vine St\\
Denver, Colorado 80208, U.S.A.}\\
E-mail: \emph{petr@math.du.edu} }

 \Received { \ November 11, 2015}

\begin{document}

\maketitle


\emph{\footnotesize  This paper was written on the occasion of the 90th anniversary of Valentin Danilovich Belousov's birthday.
Prof.\,Belousov pioneered enumerative results for quasigroups in
his book ``Fundametals of the theory of quasigroups and loops''
and his work has been a frequent source of inspiration for the
Prague algebraic school. }

\section{Introduction}\label{sec:intro}

Given an abelian group $(G,+)$, automorphisms $\varphi$, $\psi$ of $(G,+)$, and an element $c\in G$, define a new operation $*$ on $G$ by
\begin{displaymath}
    x*y = \varphi(x)+\psi(y)+c.
\end{displaymath}
The resulting quasigroup $(G,*)$ is said to be \emph{affine over} $(G,+)$, and it will be denoted by $\Q(G,+,\varphi,\psi,c)$.
Quasigroups that are affine over an abelian group are called \emph{central quasigroups} or \emph{T-quasigroups}. We will use the terms ``quasigroup affine over an abelian group'' and ``central quasigroup'' interchangeably. Central quasigroups are precisely the abelian quasigroups in the sense of universal algebra \cite{Sz}.

A quasigroup $(Q,\cdot)$ is called \emph{medial} if it satisfies the medial law
\begin{displaymath}
    (x\cdot y)\cdot(u\cdot v)=(x\cdot u)\cdot(y\cdot v).
\end{displaymath}
Medial quasigroups are also known as \emph{entropic} quasigroups.
The fundamental Toyoda-Bruck theorem \cite[Theorem 3.1]{Sta-latin} states that, up to isomorphism, medial quasigroups are precisely central quasigroups $\Q(G,+,\varphi,\psi,c)$ with commuting automorphisms $\varphi$, $\psi$.

\medskip
The classification of central (or medial) quasigroups up to isotopy is trivial in the sense that it coincides with the classification of abelian groups up to isomorphism. Indeed:
\begin{itemize}
\item If $(G,*) = \Q(G,+,\varphi,\psi,c)$ is a central quasigroup then $(G,*)$ is isotopic to $(G,+)$ via the isotopism $(x\mapsto \varphi(x)$, $x\mapsto \psi(x)+c$, $x\mapsto x)$.
\item If two central quasigroups $Q_i = \Q(G_i,+_i,\varphi_i,\psi_i,c_i)$ are isotopic then the underlying groups $(G_i,+_i)$ are isotopic. But isotopic groups are necessarily isomorphic, cf.\;\cite[Proposition 1.4]{Smi-book}.
\end{itemize}
Classifying and enumerating central and medial quasigroups up to isomorphism is nontrivial, however, and that is the topic of the present paper.

\medskip
There are not many results in the literature concerning enumeration and classification of central and medial quasigroups.

Simple idempotent medial quasigroups were classified by Smith in \cite[Theorem 6.1]{Smi}. Sokhatsky and Syvakivskij \cite{SS} classified $n$-ary quasigroups affine over cyclic groups and obtained a formula for the number of those of prime order. Kirnasovsky \cite{Kir} carried out a computer enumeration of central quasigroups up to order $15$, and obtained more classification results in his PhD thesis \cite{Kir1}. Idempotent medial quasigroups of order $p^k$, $k\leq4$, were classified by Hou \cite[Table 1]{Hou}.

At the time of writing this paper, the On-line Encyclopedia of Integer Sequences \cite{OEIS} gives the number of medial quasigroups of order $\le 8$ up to isomorphism as the sequence A226193, and there appears to be no entry for the number of central quasigroups up to isomorphism.

Dr\'apal \cite{Dra} and Sokhatsky \cite{Sok} obtained a general isomorphism theorem for quasigroups isotopic to groups, cf.\;\cite[Theorem 2.10]{Dra} and \cite[Corollary 28]{Sok}, and for central quasigroups in particular, cf.\; \cite[Theorem 3.2]{Dra}, or its restatement, Theorem \ref{Th:Alg}.
Dr\'apal applied the machinery to calculate isomorphism classes of quasigroups of order $4$ (by hand), and Kirnasovsky used Sokhatsky's theory for the calculations mentioned above.
In the present paper, we use a similar approach to obtain stronger enumeration results, taking advantages of the computer system \texttt{GAP} \cite{GAP}.


\medskip
We refer the reader to \cite{Smi-book} for general theory of quasigroups, to \cite{Dra} for a more extensive list of references on central quasigroups, to \cite{Sok} for results on quasigroups isotopic to groups, to \cite{Sta-latin} for results on quasigroups affine over various kinds of loops, and to \cite{SSta,Sz} for a broader context on affine representation of general algebraic structures. The article \cite{HR} gives a gentle introduction into automorphism groups of finite abelian groups and points to original sources on that topic.

\medskip
The paper is organized as follows.

In Section \ref{Sc:Iso}, we formulate an isomorphism theorem for central quasigroups, Theorem \ref{thm:iso}, which is less general than \cite[Theorem 2.10]{Dra} or \cite[Corollary 28]{Sok}, and equivalent to but less technical than \cite[Theorem 3.2]{Dra}. We also present the enumeration algorithm in detail.

In Section \ref{Sc:Cyclic}, we establish our own version of \cite[Theorem 2]{SS} and \cite[Theorem 3.5]{Dra} for cyclic $p$-groups, Theorem \ref{Th:Cp}, providing an explicit formula for the number of isomorphism classes. We were informed that the same result was obtained by Kirnasovsky in his unpublished PhD thesis \cite{Kir1}. Since the automorphism groups of cyclic groups are commutative, Theorem \ref{Th:Cp} also yields the number of medial quasigroups up to isomorphism over finite cyclic groups, and of prime order in particular.

Finally, the results of the enumeration are presented in the Appendix.

\section{Isomorphism theorem and enumeration algorithm}\label{Sc:Iso}

\subsection{Elementary properties of the counting functions $cq$ and $mq$}\label{Ss:Elementary}

For an abelian group $G$, let $cq(G)$ (resp. $mq(G)$) denote the number of all central (resp. medial) quasigroups over $G$ up to isomorphism. For $n\geq1$, let $cq(n)$ (resp. $mq(n)$) denote the number of all central (resp. medial) quasigroups of order $n$ up to isomorphism.

Let us establish two fundamental properties of the counting functions.

First, by the remarks in the introduction,
\begin{equation*}
    cq(n)=\sum_{|G|=n} cq(G)\quad\text{and}\quad mq(n)=\sum_{|G|=n} mq(G),
\end{equation*}
where the summations run over all abelian groups of order $n$ up to isomorphism.

Second, Proposition \ref{Pr:HK} shows that the classification of central and medial quasigroups can be reduced to prime power orders. As far as enumeration is concerned, Proposition \ref{Pr:HK} implies that the functions $cq$, $mq:\mathbb N^+\to\mathbb N^+$ are multiplicative in the number-theoretic sense.

\begin{proposition}\label{Pr:HK}
Let $G = H\times K$ be an abelian group such that $\gcd(|H|,|K|)=1$. Up to isomorphism, any quasigroup affine over $G$ can be expressed in a unique way as a direct product of a quasigroup affine over $H$ and a quasigroup affine over $K$. In particular,
\begin{equation*}
    cq(G) = cq(H)\cdot cq(K)\quad\text{and}\quad mq(G) = mq(H)\cdot mq(K).
\end{equation*}
\end{proposition}
\begin{proof}
Any automorphism of $G$ decomposes uniquely as a direct product of an automorphism of $H$ and an automorphism of $K$, cf.\;\cite[Lemma 2.1]{HR}. The rest is easy.
\end{proof}

\subsection{The isomorphism problem for central quasigroups}

Let us now consider the isomorphism problem for quasigroups affine over a fixed abelian group $(G,+)$.

Consider any group $A$. (Later we will take $A=\aut{G,+}$.) Then $A$ acts on itself by conjugation, and $A$ also acts on $A\times A$ by a simultaneous conjugation in both coordinates, i.e., $(\alpha,\beta)^\gamma = (\alpha^\gamma,\beta^\gamma)$.

\begin{lemma}\label{Lm:Orbits}
Let $A$ be a group. Let $X$ be a complete set of orbit representatives of the conjugation action of $A$ on itself. For $\xi\in X$, let $Y_\xi$ be a complete set of orbit representatives of the conjugation action of the centralizer $C_A(\xi)$ on $A$. Then
\begin{displaymath}
    \{(\xi,\upsilon):\xi\in X,\,\upsilon\in Y_\xi\}
\end{displaymath}
is a complete set of orbit representatives of the conjugation action of $A$ on $A\times A$.
\end{lemma}
\begin{proof}
For every $(\alpha,\beta)\in A\times A$ there is a unique $\xi\in X$ and some $\gamma\in A$ such that $(\alpha,\beta)$ and $(\xi,\gamma)$ are in the same orbit. For a fixed $\xi\in X$ and some $\beta$, $\gamma\in A$, we have $(\xi,\beta)$ in the same orbit as $(\xi,\gamma)$ if and only if there is $\delta\in C_A(\xi)$ such that $\beta^\delta = \gamma$.
\end{proof}

\begin{lemma}\label{Lm:Action}
Let $(G,+)$ be an abelian group, $A=\aut{G,+}$ and $\alpha$, $\beta\in A$. Then $C_A(\alpha)\cap C_A(\beta)$ acts naturally on $G/\img{1-\alpha-\beta}$.
\end{lemma}
\begin{proof}
Let $U=\img{1-\alpha-\beta}$. It suffices to show that for every $\gamma\in C_A(\alpha)\cap C_A(\beta)$ the mapping $u+U\mapsto \gamma(u)+U$ is well-defined. Now, if $u+U=v+U$ then $u=v+w-\alpha(w)-\beta(w)$ for some $w\in G$ and we have $\gamma(u) = \gamma(v) + \gamma(w) - \gamma\alpha(w) - \gamma\beta(w) = \gamma(v) + \gamma(w) - \alpha\gamma(w) - \beta\gamma(w) = \gamma(v) + (1-\alpha-\beta)(\gamma(w)) \in \gamma(v) + U$.
\end{proof}

We will now state a theorem that solves the isomorphism problem for central and medial quasigroups over $(G,+)$. Instead of showing how it follows from the more general \cite[Theorem 2.10]{Dra}, we show that it is equivalent to \cite[Theorem 3.2]{Dra}, which we restate as Theorem \ref{Th:Alg} here.

\begin{theorem}[Isomorphism problem for central quasigroups]\label{thm:iso}
Let $(G,+)$ be an abelian group, let $\varphi_1$, $\psi_1$, $\varphi_2$, $\psi_2\in\aut{G,+}$, and let $c_1$, $c_2\in G$. Then the following statements are equivalent:
\begin{enumerate}
\item[(i)] the central quasigroups $\Q(G,+,\varphi_1,\psi_1,c_1)$ and $\Q(G,+,\varphi_2,\psi_2,c_2)$ are isomorphic;
\item[(ii)] there is an automorphism $\gamma$ of $(G,+)$ and an element $u\in\mathrm{Im}(1-\varphi_1-\psi_1)$ such that
\begin{displaymath}
    \varphi_2 = \gamma \varphi_1 \gamma^{-1},\quad \psi_2 = \gamma \psi_1 \gamma^{-1},\quad c_2=\gamma(c_1+u).
\end{displaymath}
\end{enumerate}
\end{theorem}

\begin{theorem}[{{\cite[Theorem 3.2]{Dra}}}]\label{Th:Alg}
Let $(G,+)$ be an abelian group and denote $A=\aut{G,+}$. The isomorphism classes of central quasigroups (resp. medial quasigroups) over $(G,+)$ are in one-to-one correspondence with the elements of the set
\begin{displaymath}
    \{(\varphi,\psi,c):\varphi\in X,\,\psi\in Y_\varphi,\,c\in G_{\varphi,\psi}\},
\end{displaymath}
where
\begin{itemize}
	\item $X$ is a complete set of orbit representatives of the conjugation action of $A$ on itself;
	\item $Y_\varphi$ is a complete set of orbit representatives of the conjugation action of $C_A(\varphi)$ on $A$ (resp. on $C_A(\varphi)$), for every $\varphi\in X$;
	\item $G_{\varphi,\psi}$ is a complete set of orbit representatives of the natural action of $C_A(\varphi)\cap C_A(\psi)$ on $G/\img{1-\varphi-\psi}$.
\end{itemize}
\end{theorem}

Here is a proof of the equivalence of Theorems \ref{thm:iso} and \ref{Th:Alg}: By Lemma \ref{Lm:Orbits}, we can assume that we are investigating the equivalence of two triples $(\varphi,\psi,c_1)$ and $(\varphi,\psi,c_2)$ for some $\varphi\in X$, $\psi\in Y_\varphi$ and $c_1$, $c_2\in G$. Let $U=\img{1-\varphi-\psi}$. The following conditions are then equivalent for any $\gamma\in\aut{G,+}$, using Lemma \ref{Lm:Action}: $c_2 = \gamma(c_1+u)$ for some $u\in U$, $c_2\in \gamma(c_1+U) = \gamma(c_1)+U$, $c_2+U = \gamma(c_1)+U = \gamma(c_1+U)$. This finishes the proof.

\subsection{The algorithm}\label{Sc:Alg}

Theorem \ref{Th:Alg} together with the results of Subsection \ref{Ss:Elementary} gives rise to the following algorithm that enumerates central and medial quasigroups of order $n$. In the algorithm we denote by \texttt{R(H,X)} a complete set of representatives of the (clear from context) action of \texttt{H} on \texttt{X}.

\begin{algorithm}\label{Alg:Main}\ \newline
\noindent Input: positive integer $n$

\noindent Output: $cq(n)$ and $mq(n)$

\begin{verbatim}
cqn := 0; mqn := 0;
for G in the set of abelian groups of order n up to isomorphism do
    cqG := 0; mqG := 0;
    A := automorphism group of G;
    for f in R(A,A) do
        for g in R(C_A(f),A) do
            for c in R( Intersection(C_A(f),C_A(g)), G/Im(1-f-g) ) do
                cqG := cqG + 1;
                if f*g=g*f then mqG := mqG + 1; fi;
            od;
        od;
    od;
    cqn := cqn + cqG; mqn := mqn + mqG;
od;
return cqn, mqn;
\end{verbatim}
\end{algorithm}

The algorithm was implemented in the \texttt{GAP} system \cite{GAP} in a straightforward fashion, taking advantage of some functionality of the \texttt{LOOPS} \cite{LOOPS} package. The code is available from the second author at {\tt www.math.du.edu/\textasciitilde petr}.

In small situations it is possible to directly calculate the orbits of the conjugation action of $A=\aut{G,+}$ on $A\times A$. For larger groups, it is safer (due to memory constraints) to work with one conjugacy class of $A$ at a time, as in Algorithm \ref{Th:Alg}.

Among the cases we managed to calculate, the elementary abelian group $C_2^5$ took the most effort, about $4$ hours on a standard personal computer.
It might not be difficult to calculate some of the missing entries for $mq(G)$. However, $cq(C_2^6)$, for instance, appears out of reach without further theoretical advances or more substantial computational resources.

The outcome of the calculation can be found in the Appendix.

\section{Quasigroups affine over cyclic groups}\label{Sc:Cyclic}

Let $G$ be a cyclic group. Since $\aut{G}$ is commutative, every quasigroup affine over $G$ is medial.

\begin{theorem}[{\cite[p. 70]{Kir1}}]\label{Th:Cp}
Let $p$ be a prime and $k$ a positive integer. Then
\begin{displaymath}
cq(C_{p^k})=mq(C_{p^k})=    p^{2k} + p^{2k-2} - p^{k-1} - \sum_{i=k-1}^{2k-1} p^i.
\end{displaymath}
\end{theorem}

\begin{proof}
Let $G=C_{p^k}$ and $A=\mathrm{Aut}(G)$. We will identify $A$ with the $p^k-p^{k-1}$ elements of $G^* = \{a\in G:\ p\nmid a\}$. We will follow Algorithm \ref{Alg:Main}. Since $A$ is commutative, the conjugation action is trivial and we have to consider every $(\varphi,\psi)\in A\times A$. For a fixed $(\varphi,\psi)\in A\times A$, we must consider a complete set of orbit representatives $G_{\varphi,\psi}$ of the action of $A=C_A(\varphi)\cap C_A(\psi)$ on $G/\img{1-\varphi-\psi}$. Now, $\img{1-\varphi-\psi}$ is equal to $p^iG$ if and only if $p^i\mid 1-\varphi-\psi$ and $p^{i+1}\nmid 1-\varphi-\psi$.

\emph{Case} $i=0$, i.e., \[\varphi+\psi\not\equiv1\pmod p.\]
In this case, we can take $G_{\varphi,\psi}=\{0\}$. How many such pairs $(\varphi,\psi)$ exist? First, let us count those with $\varphi\equiv1\pmod p$. Then $\psi\in G^*$ can be chosen arbitrarily, hence we have $p^{k-1}(p^k-p^{k-1})$ such pairs. Next, let us count those with $\varphi\not\equiv1\pmod p$. Then $\psi\in G^*$ must satisfy $\psi\not\equiv 1-\varphi\pmod p$, hence we have $(p^k-2p^{k-1})(p^k-2p^{k-1})$ such pairs. Since $|G_{\varphi,\psi}|=1$, this case contributes to $cq(G)$ by \[p^{k-1}(p^k-p^{k-1}) + (p^k-2p^{k-1})^2.\]

\emph{Cases} $i=1,\dots,k-1$, i.e.,
\begin{displaymath}
    \varphi+\psi\equiv1\pmod{p^i} \quad\text{and}\quad\varphi+\psi\not\equiv1\pmod{p^{i+1}}.
\end{displaymath}
In this case, we can take $G_{\varphi,\psi}=\{0,p^0,\dots,p^{i-1}\}$. How many such pairs $(\varphi,\psi)$ exist? For $\varphi\equiv1\pmod p$, any solution $\psi$ to the congruence above is divisible by $p$, hence there is no such solution $\psi\in G^*$. For $\varphi\not\equiv1\pmod p$, we have precisely $p^{k-i}-p^{k-i-1}$ solutions to the conditions in $G^*$. Since $|G_{\varphi,\psi}|=i+1$, this case contributes to $cq(G)$ by \[(p^k-2p^{k-1})(p^{k-i}-p^{k-i-1})(i+1).\]

\emph{Case} $i=k$, i.e., \[\varphi+\psi=1.\]
In this case, we can take $G_{\varphi,\psi}=\{0,p^0,\dots,p^{k-1}\}$. How many such pairs $(\varphi,\psi)$ exist? Since $\psi$ is uniquely determined by $\varphi$ and neither of $\varphi,\psi$ shall be divisible by $p$, we have precisely $p^k-2p^{k-1}$ such pairs. Since $|G_{\varphi,\psi}|=k+1$, this case contributes to $cq(G)$ by \[(p^k-2p^{k-1})(k+1).\]

Summarized, the cases $i=1,\dots,k$ contribute to $cq(G)$ the total of
\begin{displaymath}
    (p^k-2p^{k-1})\left(\left(
        \sum_{i=1}^{k-1}(p^{k-i}-p^{k-i-1})\cdot(i+1)
    \right)+(k+1)\right),
\end{displaymath}
which, after rearrangement, gives
\[(p^k-2p^{k-1})(2p^{k-1}+p^{k-2}+p^{k-3}+\cdots+p+1).\]
The total sum is then
\begin{align*}
    cq(G) &= p^{2k-1}-p^{2k-2} + (p^k-2p^{k-1})((p^k-2p^{k-1}) + (2p^{k-1}+p^{k-2}+p^{k-3}+\cdots+p+1)) \\
        &= p^{2k-1} - p^{2k-2} + (p^k-2p^{k-1})( p^k + p^{k-2} + p^{k-3} + \cdots + p+1 )\\
        &= p^{2k}-p^{2k-1} - p^{2k-3} - \cdots - p^k - 2p^{k-1},
\end{align*}
which can be expressed as in the statement of the theorem.
\end{proof}

\begin{corollary}
For any $k\ge 1$ we have $cq(C_{2^k}) = mq(C_{2^k}) = 2^{2k-2}$.
\end{corollary}

\begin{corollary}\label{c:prime}
For any prime $p$ we have $cq(p) = mq(p) = p^2-p-1$.
\end{corollary}

Corollary \ref{c:prime} is a special case of \cite[Corollary 2]{SS} for binary quasigroups.

As a counterpart to Theorem \ref{Th:Cp}, we ask:

\begin{problem}
For a prime $p$ and $k>1$, find explicit formulas for $cq(C_p^k)$ and $mq(C_p^k)$.
\end{problem}

\newpage

\section*{Appendix: Central and medial quasigroups of order less than $128$}

The following table contains the results of our enumeration of central and medial quasigroups of order less than $128$.

If a row in the table starts with $n/k$ then: column ``$G$'' gives the catalog number $n/k$ corresponding to the abelian group \texttt{SmallGroup(n,k)} of \texttt{GAP}; column ``structure'' gives a structural description of the group $G$ from which a decomposition of $G$ into $p$-primary components is readily seen and hence Proposition \ref{Pr:HK} can be routinely applied; column ``$|A|$'' gives the cardinality of the group $A=\mathrm{Aut}(G)$; column ``$|X|$'' gives the number of conjugacy classes of $A$; column ``$|O|$'' gives the number of orbits of the conjugation action of $A$ on $A\times A$ (with action $(f,g)^h = (f^h,g^h)$), which is a lower bound on the number of quasigroups affine over $G$; column ``$cq$'' gives the number of quasigroups affine over $G$ up to isomorphism; column ``$|O_c|$'' gives the number of orbits in $O$ with a representative $(f,g)$ such that $fg=gf$, which is a lower bound on the number of medial quasigroups over $G$; column ``$mq$'' gives the number of medial quasigroups over $G$ up to isomorphism; and column ``ref'' gives a reference to a numbered result within this paper if the entries in the row follow from the cited result and possibly also from previously listed table entries.

If a row in the table starts with $\mathbf{n}$ then: column ``$G$'' gives the order $n$; column ``$cq$'' gives the number of central quasigroups of order $n$ up to isomorphism; and column ``$mq$'' gives the number of medial quasigroups of order $n$ up to isomorphism.

Entries that we were not able to establish are denoted by ``?'' or ``\textbf{?}''.

All entries corresponding to prime-power orders were explicitly calculated by Algorithm \ref{Alg:Main} although the cyclic cases follow from Theorem \ref{Th:Cp}. Many of the entries corresponding to the remaining orders were also initially obtained by Algorithm \ref{Alg:Main} (to test the algorithm) but in the final version they were calculated directly from earlier entries using Proposition \ref{Pr:HK}.

To reduce the number of transcription and arithmetical errors, the entries and the \LaTeX\  source of the table were computer generated.

\begin{small}

\begin{displaymath}\
\begin{array}{|rr|rr|rr|rr|r|}
\hline
G & \text{structure} & |A| & |X| & |O| & cq & |O_c| & mq & \text{ref}\\
\hline\hline
1/1&C_{1}&1&1&1&1&1&1&\\
\textbf{1}&&&&&\textbf{1}&&\textbf{1}&\\
\hline
2/1&C_{2}&1&1&1&1&1&1&\ref{Th:Cp}\\
\textbf{2}&&&&&\textbf{1}&&\textbf{1}&\\
\hline
3/1&C_{3}&2&2&4&5&4&5&\ref{Th:Cp}\\
\textbf{3}&&&&&\textbf{5}&&\textbf{5}&\\
\hline
4/1&C_{4}&2&2&4&4&4&4&\ref{Th:Cp}\\
4/2&C_{2}^{2}&6&3&11&15&8&9&\\
\textbf{4}&&&&&\textbf{19}&&\textbf{13}&\\
\hline
5/1&C_{5}&4&4&16&19&16&19&\ref{Th:Cp}\\
\textbf{5}&&&&&\textbf{19}&&\textbf{19}&\\
\hline
6/2&C_{2}{\times} C_{3}&2&2&4&5&4&5&\ref{Pr:HK}\\
\textbf{6}&&&&&\textbf{5}&&\textbf{5}&\\
\hline
7/1&C_{7}&6&6&36&41&36&41&\ref{Th:Cp}\\
\textbf{7}&&&&&\textbf{41}&&\textbf{41}&\\
\hline
\end{array}\end{displaymath}
\newpage

\begin{displaymath}\
\begin{array}{|rr|rr|rr|rr|r|}
\hline
G & \text{structure} & |A| & |X| & |O| & cq & |O_c| & mq & \text{ref}\\
\hline\hline
8/1&C_{8}&4&4&16&16&16&16&\ref{Th:Cp}\\
8/2&C_{4}{\times} C_{2}&8&5&28&28&22&22&\\
8/5&C_{2}^{3}&168&6&197&341&32&35&\\
\textbf{8}&&&&&\textbf{385}&&\textbf{73}&\\
\hline
9/1&C_{9}&6&6&36&48&36&48&\ref{Th:Cp}\\
9/2&C_{3}^{2}&48&8&136&183&56&68&\\
\textbf{9}&&&&&\textbf{231}&&\textbf{116}&\\
\hline
10/2&C_{2}{\times} C_{5}&4&4&16&19&16&19&\ref{Pr:HK}\\
\textbf{10}&&&&&\textbf{19}&&\textbf{19}&\\
\hline
11/1&C_{11}&10&10&100&109&100&109&\ref{Th:Cp}\\
\textbf{11}&&&&&\textbf{109}&&\textbf{109}&\\
\hline
12/2&C_{4}{\times} C_{3}&4&4&16&20&16&20&\ref{Pr:HK}\\
12/5&C_{2}^{2}{\times} C_{3}&12&6&44&75&32&45&\ref{Pr:HK}\\
\textbf{12}&&&&&\textbf{95}&&\textbf{65}&\\
\hline
13/1&C_{13}&12&12&144&155&144&155&\ref{Th:Cp}\\
\textbf{13}&&&&&\textbf{155}&&\textbf{155}&\\
\hline
14/2&C_{2}{\times} C_{7}&6&6&36&41&36&41&\ref{Pr:HK}\\
\textbf{14}&&&&&\textbf{41}&&\textbf{41}&\\
\hline
15/1&C_{3}{\times} C_{5}&8&8&64&95&64&95&\ref{Pr:HK}\\
\textbf{15}&&&&&\textbf{95}&&\textbf{95}&\\
\hline
16/1&C_{16}&8&8&64&64&64&64&\ref{Th:Cp}\\
16/2&C_{4}^{2}&96&14&400&624&168&188&\\
16/5&C_{8}{\times} C_{2}&16&10&112&112&88&88&\\
16/10&C_{4}{\times} C_{2}^{2}&192&13&564&820&146&150&\\
16/14&C_{2}^{4}&20160&14&20747&39767&160&179&\\
\textbf{16}&&&&&\textbf{41387}&&\textbf{669}&\\
\hline
17/1&C_{17}&16&16&256&271&256&271&\ref{Th:Cp}\\
\textbf{17}&&&&&\textbf{271}&&\textbf{271}&\\
\hline
18/2&C_{2}{\times} C_{9}&6&6&36&48&36&48&\ref{Pr:HK}\\
18/5&C_{2}{\times} C_{3}^{2}&48&8&136&183&56&68&\ref{Pr:HK}\\
\textbf{18}&&&&&\textbf{231}&&\textbf{116}&\\
\hline
19/1&C_{19}&18&18&324&341&324&341&\ref{Th:Cp}\\
\textbf{19}&&&&&\textbf{341}&&\textbf{341}&\\
\hline
20/2&C_{4}{\times} C_{5}&8&8&64&76&64&76&\ref{Pr:HK}\\
20/5&C_{2}^{2}{\times} C_{5}&24&12&176&285&128&171&\ref{Pr:HK}\\
\textbf{20}&&&&&\textbf{361}&&\textbf{247}&\\
\hline
21/2&C_{3}{\times} C_{7}&12&12&144&205&144&205&\ref{Pr:HK}\\
\textbf{21}&&&&&\textbf{205}&&\textbf{205}&\\
\hline
22/2&C_{2}{\times} C_{11}&10&10&100&109&100&109&\ref{Pr:HK}\\
\textbf{22}&&&&&\textbf{109}&&\textbf{109}&\\
\hline
23/1&C_{23}&22&22&484&505&484&505&\ref{Th:Cp}\\
\textbf{23}&&&&&\textbf{505}&&\textbf{505}&\\
\hline
\end{array}
\end{displaymath}
\newpage

\begin{displaymath}\
\begin{array}{|rr|rr|rr|rr|r|}
\hline
G & \text{structure} & |A| & |X| & |O| & cq & |O_c| & mq & \text{ref}\\
\hline\hline
24/2&C_{8}{\times} C_{3}&8&8&64&80&64&80&\ref{Pr:HK}\\
24/9&C_{4}{\times} C_{2}{\times} C_{3}&16&10&112&140&88&110&\ref{Pr:HK}\\
24/15&C_{2}^{3}{\times} C_{3}&336&12&788&1705&128&175&\ref{Pr:HK}\\
\textbf{24}&&&&&\textbf{1925}&&\textbf{365}&\\
\hline
25/1&C_{25}&20&20&400&490&400&490&\ref{Th:Cp}\\
25/2&C_{5}^{2}&480&24&2336&2847&512&594&\\
\textbf{25}&&&&&\textbf{3337}&&\textbf{1084}&\\
\hline
26/2&C_{2}{\times} C_{13}&12&12&144&155&144&155&\ref{Pr:HK}\\
\textbf{26}&&&&&\textbf{155}&&\textbf{155}&\\
\hline
27/1&C_{27}&18&18&324&441&324&441&\ref{Th:Cp}\\
27/2&C_{9}{\times} C_{3}&108&20&864&1356&336&528&\\
27/5&C_{3}^{3}&11232&24&23236&34321&484&605&\\
\textbf{27}&&&&&\textbf{36118}&&\textbf{1574}&\\
\hline
28/2&C_{4}{\times} C_{7}&12&12&144&164&144&164&\ref{Pr:HK}\\
28/4&C_{2}^{2}{\times} C_{7}&36&18&396&615&288&369&\ref{Pr:HK}\\
\textbf{28}&&&&&\textbf{779}&&\textbf{533}&\\
\hline
29/1&C_{29}&28&28&784&811&784&811&\ref{Th:Cp}\\
\textbf{29}&&&&&\textbf{811}&&\textbf{811}&\\
\hline
30/4&C_{2}{\times} C_{3}{\times} C_{5}&8&8&64&95&64&95&\ref{Pr:HK}\\
\textbf{30}&&&&&\textbf{95}&&\textbf{95}&\\
\hline
31/1&C_{31}&30&30&900&929&900&929&\ref{Th:Cp}\\
\textbf{31}&&&&&\textbf{929}&&\textbf{929}&\\
\hline
32/1&C_{32}&16&16&256&256&256&256&\ref{Th:Cp}\\
32/3&C_{8}{\times} C_{4}&128&26&1216&1216&592&592&\\
32/16&C_{16}{\times} C_{2}&32&20&448&448&352&352&\\
32/21&C_{4}^{2}{\times} C_{2}&1536&30&6224&9808&884&904&\\
32/36&C_{8}{\times} C_{2}^{2}&384&26&2256&3280&584&600&\\
32/45&C_{4}{\times} C_{2}^{3}&21504&30&48412&87580&804&834&\\
32/51&C_{2}^{5}&9999360&27&10024077&19721077&590&655&\\
\textbf{32}&&&&&\textbf{19823665}&&\textbf{4193}&\\
\hline
33/1&C_{3}{\times} C_{11}&20&20&400&545&400&545&\ref{Pr:HK}\\
\textbf{33}&&&&&\textbf{545}&&\textbf{545}&\\
\hline
34/2&C_{2}{\times} C_{17}&16&16&256&271&256&271&\ref{Pr:HK}\\
\textbf{34}&&&&&\textbf{271}&&\textbf{271}&\\
\hline
35/1&C_{5}{\times} C_{7}&24&24&576&779&576&779&\ref{Pr:HK}\\
\textbf{35}&&&&&\textbf{779}&&\textbf{779}&\\
\hline
36/2&C_{4}{\times} C_{9}&12&12&144&192&144&192&\ref{Pr:HK}\\
36/5&C_{2}^{2}{\times} C_{9}&36&18&396&720&288&432&\ref{Pr:HK}\\
36/8&C_{4}{\times} C_{3}^{2}&96&16&544&732&224&272&\ref{Pr:HK}\\
36/14&C_{2}^{2}{\times} C_{3}^{2}&288&24&1496&2745&448&612&\ref{Pr:HK}\\
\textbf{36}&&&&&\textbf{4389}&&\textbf{1508}&\\
\hline
37/1&C_{37}&36&36&1296&1331&1296&1331&\ref{Th:Cp}\\
\textbf{37}&&&&&\textbf{1331}&&\textbf{1331}&\\
\hline
\end{array}
\end{displaymath}
\newpage

\begin{displaymath}\
\begin{array}{|rr|rr|rr|rr|r|}
\hline
G & \text{structure} & |A| & |X| & |O| & cq & |O_c| & mq & \text{ref}\\
\hline\hline
38/2&C_{2}{\times} C_{19}&18&18&324&341&324&341&\ref{Pr:HK}\\
\textbf{38}&&&&&\textbf{341}&&\textbf{341}&\\
\hline
39/2&C_{3}{\times} C_{13}&24&24&576&775&576&775&\ref{Pr:HK}\\
\textbf{39}&&&&&\textbf{775}&&\textbf{775}&\\
\hline
40/2&C_{8}{\times} C_{5}&16&16&256&304&256&304&\ref{Pr:HK}\\
40/9&C_{4}{\times} C_{2}{\times} C_{5}&32&20&448&532&352&418&\ref{Pr:HK}\\
40/14&C_{2}^{3}{\times} C_{5}&672&24&3152&6479&512&665&\ref{Pr:HK}\\
\textbf{40}&&&&&\textbf{7315}&&\textbf{1387}&\\
\hline
41/1&C_{41}&40&40&1600&1639&1600&1639&\ref{Th:Cp}\\
\textbf{41}&&&&&\textbf{1639}&&\textbf{1639}&\\
\hline
42/6&C_{2}{\times} C_{3}{\times} C_{7}&12&12&144&205&144&205&\ref{Pr:HK}\\
\textbf{42}&&&&&\textbf{205}&&\textbf{205}&\\
\hline
43/1&C_{43}&42&42&1764&1805&1764&1805&\ref{Th:Cp}\\
\textbf{43}&&&&&\textbf{1805}&&\textbf{1805}&\\
\hline
44/2&C_{4}{\times} C_{11}&20&20&400&436&400&436&\ref{Pr:HK}\\
44/4&C_{2}^{2}{\times} C_{11}&60&30&1100&1635&800&981&\ref{Pr:HK}\\
\textbf{44}&&&&&\textbf{2071}&&\textbf{1417}&\\
\hline
45/1&C_{9}{\times} C_{5}&24&24&576&912&576&912&\ref{Pr:HK}\\
45/2&C_{3}^{2}{\times} C_{5}&192&32&2176&3477&896&1292&\ref{Pr:HK}\\
\textbf{45}&&&&&\textbf{4389}&&\textbf{2204}&\\
\hline
46/2&C_{2}{\times} C_{23}&22&22&484&505&484&505&\ref{Pr:HK}\\
\textbf{46}&&&&&\textbf{505}&&\textbf{505}&\\
\hline
47/1&C_{47}&46&46&2116&2161&2116&2161&\ref{Th:Cp}\\
\textbf{47}&&&&&\textbf{2161}&&\textbf{2161}&\\
\hline
48/2&C_{16}{\times} C_{3}&16&16&256&320&256&320&\ref{Pr:HK}\\
48/20&C_{4}^{2}{\times} C_{3}&192&28&1600&3120&672&940&\ref{Pr:HK}\\
48/23&C_{8}{\times} C_{2}{\times} C_{3}&32&20&448&560&352&440&\ref{Pr:HK}\\
48/44&C_{4}{\times} C_{2}^{2}{\times} C_{3}&384&26&2256&4100&584&750&\ref{Pr:HK}\\
48/52&C_{2}^{4}{\times} C_{3}&40320&28&82988&198835&640&895&\ref{Pr:HK}\\
\textbf{48}&&&&&\textbf{206935}&&\textbf{3345}&\\
\hline
49/1&C_{49}&42&42&1764&2044&1764&2044&\ref{Th:Cp}\\
49/2&C_{7}^{2}&2016&48&13896&16055&2088&2344&\\
\textbf{49}&&&&&\textbf{18099}&&\textbf{4388}&\\
\hline
50/2&C_{2}{\times} C_{25}&20&20&400&490&400&490&\ref{Pr:HK}\\
50/5&C_{2}{\times} C_{5}^{2}&480&24&2336&2847&512&594&\ref{Pr:HK}\\
\textbf{50}&&&&&\textbf{3337}&&\textbf{1084}&\\
\hline
51/1&C_{3}{\times} C_{17}&32&32&1024&1355&1024&1355&\ref{Pr:HK}\\
\textbf{51}&&&&&\textbf{1355}&&\textbf{1355}&\\
\hline
52/2&C_{4}{\times} C_{13}&24&24&576&620&576&620&\ref{Pr:HK}\\
52/5&C_{2}^{2}{\times} C_{13}&72&36&1584&2325&1152&1395&\ref{Pr:HK}\\
\textbf{52}&&&&&\textbf{2945}&&\textbf{2015}&\\
\hline
53/1&C_{53}&52&52&2704&2755&2704&2755&\ref{Th:Cp}\\
\textbf{53}&&&&&\textbf{2755}&&\textbf{2755}&\\
\hline
\end{array}
\end{displaymath}
\newpage

\begin{displaymath}\
\begin{array}{|rr|rr|rr|rr|r|}
\hline
G & \text{structure} & |A| & |X| & |O| & cq & |O_c| & mq & \text{ref}\\
\hline\hline
54/2&C_{2}{\times} C_{27}&18&18&324&441&324&441&\ref{Pr:HK}\\
54/9&C_{2}{\times} C_{9}{\times} C_{3}&108&20&864&1356&336&528&\ref{Pr:HK}\\
54/15&C_{2}{\times} C_{3}^{3}&11232&24&23236&34321&484&605&\ref{Pr:HK}\\
\textbf{54}&&&&&\textbf{36118}&&\textbf{1574}&\\
\hline
55/2&C_{5}{\times} C_{11}&40&40&1600&2071&1600&2071&\ref{Pr:HK}\\
\textbf{55}&&&&&\textbf{2071}&&\textbf{2071}&\\
\hline
56/2&C_{8}{\times} C_{7}&24&24&576&656&576&656&\ref{Pr:HK}\\
56/8&C_{4}{\times} C_{2}{\times} C_{7}&48&30&1008&1148&792&902&\ref{Pr:HK}\\
56/13&C_{2}^{3}{\times} C_{7}&1008&36&7092&13981&1152&1435&\ref{Pr:HK}\\
\textbf{56}&&&&&\textbf{15785}&&\textbf{2993}&\\
\hline
57/2&C_{3}{\times} C_{19}&36&36&1296&1705&1296&1705&\ref{Pr:HK}\\
\textbf{57}&&&&&\textbf{1705}&&\textbf{1705}&\\
\hline
58/2&C_{2}{\times} C_{29}&28&28&784&811&784&811&\ref{Pr:HK}\\
\textbf{58}&&&&&\textbf{811}&&\textbf{811}&\\
\hline
59/1&C_{59}&58&58&3364&3421&3364&3421&\ref{Th:Cp}\\
\textbf{59}&&&&&\textbf{3421}&&\textbf{3421}&\\
\hline
60/4&C_{4}{\times} C_{3}{\times} C_{5}&16&16&256&380&256&380&\ref{Pr:HK}\\
60/13&C_{2}^{2}{\times} C_{3}{\times} C_{5}&48&24&704&1425&512&855&\ref{Pr:HK}\\
\textbf{60}&&&&&\textbf{1805}&&\textbf{1235}&\\
\hline
61/1&C_{61}&60&60&3600&3659&3600&3659&\ref{Th:Cp}\\
\textbf{61}&&&&&\textbf{3659}&&\textbf{3659}&\\
\hline
62/2&C_{2}{\times} C_{31}&30&30&900&929&900&929&\ref{Pr:HK}\\
\textbf{62}&&&&&\textbf{929}&&\textbf{929}&\\
\hline
63/2&C_{9}{\times} C_{7}&36&36&1296&1968&1296&1968&\ref{Pr:HK}\\
63/4&C_{3}^{2}{\times} C_{7}&288&48&4896&7503&2016&2788&\ref{Pr:HK}\\
\textbf{63}&&&&&\textbf{9471}&&\textbf{4756}&\\
\hline
64/1&C_{64}&32&32&1024&1024&1024&1024&\ref{Th:Cp}\\
64/2&C_{8}^{2}&1536&60&13568&22784&3072&3408&\\
64/26&C_{16}{\times} C_{4}&256&52&4864&4864&2368&2368&\\
64/50&C_{32}{\times} C_{2}&64&40&1792&1792&1408&1408&\\
64/55&C_{4}^{3}&86016&60&206144&441664&4448&4672&\\
64/83&C_{8}{\times} C_{4}{\times} C_{2}&2048&104&31168&31168&7240&7240&\\
64/183&C_{16}{\times} C_{2}^{2}&768&52&9024&13120&2336&2400&\\
64/192&C_{4}^{2}{\times} C_{2}^{2}&147456&100&550480&1239472&9108&9656&\\
64/246&C_{8}{\times} C_{2}^{3}&43008&60&193648&350320&3216&3336&\\
64/260&C_{4}{\times} C_{2}^{4}&10321920&69&?&?&?&?&\\
64/267&C_{2}^{6}&20158709760&60&?&?&?&?&\\
\textbf{64}&&&&&\textbf{?}&&\textbf{?}&\\
\hline
65/1&C_{5}{\times} C_{13}&48&48&2304&2945&2304&2945&\ref{Pr:HK}\\
\textbf{65}&&&&&\textbf{2945}&&\textbf{2945}&\\
\hline
66/4&C_{2}{\times} C_{3}{\times} C_{11}&20&20&400&545&400&545&\ref{Pr:HK}\\
\textbf{66}&&&&&\textbf{545}&&\textbf{545}&\\
\hline
67/1&C_{67}&66&66&4356&4421&4356&4421&\ref{Th:Cp}\\
\textbf{67}&&&&&\textbf{4421}&&\textbf{4421}&\\
\hline
\end{array}
\end{displaymath}
\newpage

\begin{displaymath}\
\begin{array}{|rr|rr|rr|rr|r|}
\hline
G & \text{structure} & |A| & |X| & |O| & cq & |O_c| & mq & \text{ref}\\
\hline\hline
68/2&C_{4}{\times} C_{17}&32&32&1024&1084&1024&1084&\ref{Pr:HK}\\
68/5&C_{2}^{2}{\times} C_{17}&96&48&2816&4065&2048&2439&\ref{Pr:HK}\\
\textbf{68}&&&&&\textbf{5149}&&\textbf{3523}&\\
\hline
69/1&C_{3}{\times} C_{23}&44&44&1936&2525&1936&2525&\ref{Pr:HK}\\
\textbf{69}&&&&&\textbf{2525}&&\textbf{2525}&\\
\hline
70/4&C_{2}{\times} C_{5}{\times} C_{7}&24&24&576&779&576&779&\ref{Pr:HK}\\
\textbf{70}&&&&&\textbf{779}&&\textbf{779}&\\
\hline
71/1&C_{71}&70&70&4900&4969&4900&4969&\ref{Th:Cp}\\
\textbf{71}&&&&&\textbf{4969}&&\textbf{4969}&\\
\hline
72/2&C_{8}{\times} C_{9}&24&24&576&768&576&768&\ref{Pr:HK}\\
72/9&C_{4}{\times} C_{2}{\times} C_{9}&48&30&1008&1344&792&1056&\ref{Pr:HK}\\
72/14&C_{8}{\times} C_{3}^{2}&192&32&2176&2928&896&1088&\ref{Pr:HK}\\
72/18&C_{2}^{3}{\times} C_{9}&1008&36&7092&16368&1152&1680&\ref{Pr:HK}\\
72/36&C_{4}{\times} C_{2}{\times} C_{3}^{2}&384&40&3808&5124&1232&1496&\ref{Pr:HK}\\
72/50&C_{2}^{3}{\times} C_{3}^{2}&8064&48&26792&62403&1792&2380&\ref{Pr:HK}\\
\textbf{72}&&&&&\textbf{88935}&&\textbf{8468}&\\
\hline
73/1&C_{73}&72&72&5184&5255&5184&5255&\ref{Th:Cp}\\
\textbf{73}&&&&&\textbf{5255}&&\textbf{5255}&\\
\hline
74/2&C_{2}{\times} C_{37}&36&36&1296&1331&1296&1331&\ref{Pr:HK}\\
\textbf{74}&&&&&\textbf{1331}&&\textbf{1331}&\\
\hline
75/1&C_{3}{\times} C_{25}&40&40&1600&2450&1600&2450&\ref{Pr:HK}\\
75/3&C_{3}{\times} C_{5}^{2}&960&48&9344&14235&2048&2970&\ref{Pr:HK}\\
\textbf{75}&&&&&\textbf{16685}&&\textbf{5420}&\\
\hline
76/2&C_{4}{\times} C_{19}&36&36&1296&1364&1296&1364&\ref{Pr:HK}\\
76/4&C_{2}^{2}{\times} C_{19}&108&54&3564&5115&2592&3069&\ref{Pr:HK}\\
\textbf{76}&&&&&\textbf{6479}&&\textbf{4433}&\\
\hline
77/1&C_{7}{\times} C_{11}&60&60&3600&4469&3600&4469&\ref{Pr:HK}\\
\textbf{77}&&&&&\textbf{4469}&&\textbf{4469}&\\
\hline
78/6&C_{2}{\times} C_{3}{\times} C_{13}&24&24&576&775&576&775&\ref{Pr:HK}\\
\textbf{78}&&&&&\textbf{775}&&\textbf{775}&\\
\hline
79/1&C_{79}&78&78&6084&6161&6084&6161&\ref{Th:Cp}\\
\textbf{79}&&&&&\textbf{6161}&&\textbf{6161}&\\
\hline
80/2&C_{16}{\times} C_{5}&32&32&1024&1216&1024&1216&\ref{Pr:HK}\\
80/20&C_{4}^{2}{\times} C_{5}&384&56&6400&11856&2688&3572&\ref{Pr:HK}\\
80/23&C_{8}{\times} C_{2}{\times} C_{5}&64&40&1792&2128&1408&1672&\ref{Pr:HK}\\
80/45&C_{4}{\times} C_{2}^{2}{\times} C_{5}&768&52&9024&15580&2336&2850&\ref{Pr:HK}\\
80/52&C_{2}^{4}{\times} C_{5}&80640&56&331952&755573&2560&3401&\ref{Pr:HK}\\
\textbf{80}&&&&&\textbf{786353}&&\textbf{12711}&\\
\hline
81/1&C_{81}&54&54&2916&3996&2916&3996&\ref{Th:Cp}\\
81/2&C_{9}^{2}&3888&78&35316&54405&5616&8055&\\
81/5&C_{27}{\times} C_{3}&324&60&7776&12897&3024&5157&\\
81/11&C_{9}{\times} C_{3}^{2}&23328&74&152892&270441&4176&7167&\\
81/15&C_{3}^{4}&24261120&78&?&?&?&?&\\
\textbf{81}&&&&&\textbf{?}&&\textbf{?}&\\
\hline
\end{array}
\end{displaymath}
\newpage

\begin{displaymath}\
\begin{array}{|rr|rr|rr|rr|r|}
\hline
G & \text{structure} & |A| & |X| & |O| & cq & |O_c| & mq & \text{ref}\\
\hline\hline
82/2&C_{2}{\times} C_{41}&40&40&1600&1639&1600&1639&\ref{Pr:HK}\\
\textbf{82}&&&&&\textbf{1639}&&\textbf{1639}&\\
\hline
83/1&C_{83}&82&82&6724&6805&6724&6805&\ref{Th:Cp}\\
\textbf{83}&&&&&\textbf{6805}&&\textbf{6805}&\\
\hline
84/6&C_{4}{\times} C_{3}{\times} C_{7}&24&24&576&820&576&820&\ref{Pr:HK}\\
84/15&C_{2}^{2}{\times} C_{3}{\times} C_{7}&72&36&1584&3075&1152&1845&\ref{Pr:HK}\\
\textbf{84}&&&&&\textbf{3895}&&\textbf{2665}&\\
\hline
85/1&C_{5}{\times} C_{17}&64&64&4096&5149&4096&5149&\ref{Pr:HK}\\
\textbf{85}&&&&&\textbf{5149}&&\textbf{5149}&\\
\hline
86/2&C_{2}{\times} C_{43}&42&42&1764&1805&1764&1805&\ref{Pr:HK}\\
\textbf{86}&&&&&\textbf{1805}&&\textbf{1805}&\\
\hline
87/1&C_{3}{\times} C_{29}&56&56&3136&4055&3136&4055&\ref{Pr:HK}\\
\textbf{87}&&&&&\textbf{4055}&&\textbf{4055}&\\
\hline
88/2&C_{8}{\times} C_{11}&40&40&1600&1744&1600&1744&\ref{Pr:HK}\\
88/8&C_{4}{\times} C_{2}{\times} C_{11}&80&50&2800&3052&2200&2398&\ref{Pr:HK}\\
88/12&C_{2}^{3}{\times} C_{11}&1680&60&19700&37169&3200&3815&\ref{Pr:HK}\\
\textbf{88}&&&&&\textbf{41965}&&\textbf{7957}&\\
\hline
89/1&C_{89}&88&88&7744&7831&7744&7831&\ref{Th:Cp}\\
\textbf{89}&&&&&\textbf{7831}&&\textbf{7831}&\\
\hline
90/4&C_{2}{\times} C_{9}{\times} C_{5}&24&24&576&912&576&912&\ref{Pr:HK}\\
90/10&C_{2}{\times} C_{3}^{2}{\times} C_{5}&192&32&2176&3477&896&1292&\ref{Pr:HK}\\
\textbf{90}&&&&&\textbf{4389}&&\textbf{2204}&\\
\hline
91/1&C_{7}{\times} C_{13}&72&72&5184&6355&5184&6355&\ref{Pr:HK}\\
\textbf{91}&&&&&\textbf{6355}&&\textbf{6355}&\\
\hline
92/2&C_{4}{\times} C_{23}&44&44&1936&2020&1936&2020&\ref{Pr:HK}\\
92/4&C_{2}^{2}{\times} C_{23}&132&66&5324&7575&3872&4545&\ref{Pr:HK}\\
\textbf{92}&&&&&\textbf{9595}&&\textbf{6565}&\\
\hline
93/2&C_{3}{\times} C_{31}&60&60&3600&4645&3600&4645&\ref{Pr:HK}\\
\textbf{93}&&&&&\textbf{4645}&&\textbf{4645}&\\
\hline
94/2&C_{2}{\times} C_{47}&46&46&2116&2161&2116&2161&\ref{Pr:HK}\\
\textbf{94}&&&&&\textbf{2161}&&\textbf{2161}&\\
\hline
95/1&C_{5}{\times} C_{19}&72&72&5184&6479&5184&6479&\ref{Pr:HK}\\
\textbf{95}&&&&&\textbf{6479}&&\textbf{6479}&\\
\hline
96/2&C_{32}{\times} C_{3}&32&32&1024&1280&1024&1280&\ref{Pr:HK}\\
96/46&C_{8}{\times} C_{4}{\times} C_{3}&256&52&4864&6080&2368&2960&\ref{Pr:HK}\\
96/59&C_{16}{\times} C_{2}{\times} C_{3}&64&40&1792&2240&1408&1760&\ref{Pr:HK}\\
96/161&C_{4}^{2}{\times} C_{2}{\times} C_{3}&3072&60&24896&49040&3536&4520&\ref{Pr:HK}\\
96/176&C_{8}{\times} C_{2}^{2}{\times} C_{3}&768&52&9024&16400&2336&3000&\ref{Pr:HK}\\
96/220&C_{4}{\times} C_{2}^{3}{\times} C_{3}&43008&60&193648&437900&3216&4170&\ref{Pr:HK}\\
96/231&C_{2}^{5}{\times} C_{3}&19998720&54&40096308&98605385&2360&3275&\ref{Pr:HK}\\
\textbf{96}&&&&&\textbf{99118325}&&\textbf{20965}&\\
\hline
97/1&C_{97}&96&96&9216&9311&9216&9311&\ref{Th:Cp}\\
\textbf{97}&&&&&\textbf{9311}&&\textbf{9311}&\\
\hline
\end{array}
\end{displaymath}
\newpage

\begin{displaymath}\
\begin{array}{|rr|rr|rr|rr|r|}
\hline
G & \text{structure} & |A| & |X| & |O| & cq & |O_c| & mq & \text{ref}\\
\hline\hline
98/2&C_{2}{\times} C_{49}&42&42&1764&2044&1764&2044&\ref{Pr:HK}\\
98/5&C_{2}{\times} C_{7}^{2}&2016&48&13896&16055&2088&2344&\ref{Pr:HK}\\
\textbf{98}&&&&&\textbf{18099}&&\textbf{4388}&\\
\hline
99/1&C_{9}{\times} C_{11}&60&60&3600&5232&3600&5232&\ref{Pr:HK}\\
99/2&C_{3}^{2}{\times} C_{11}&480&80&13600&19947&5600&7412&\ref{Pr:HK}\\
\textbf{99}&&&&&\textbf{25179}&&\textbf{12644}&\\
\hline
100/2&C_{4}{\times} C_{25}&40&40&1600&1960&1600&1960&\ref{Pr:HK}\\
100/5&C_{2}^{2}{\times} C_{25}&120&60&4400&7350&3200&4410&\ref{Pr:HK}\\
100/8&C_{4}{\times} C_{5}^{2}&960&48&9344&11388&2048&2376&\ref{Pr:HK}\\
100/16&C_{2}^{2}{\times} C_{5}^{2}&2880&72&25696&42705&4096&5346&\ref{Pr:HK}\\
\textbf{100}&&&&&\textbf{63403}&&\textbf{14092}&\\
\hline
101/1&C_{101}&100&100&10000&10099&10000&10099&\ref{Th:Cp}\\
\textbf{101}&&&&&\textbf{10099}&&\textbf{10099}&\\
\hline
102/4&C_{2}{\times} C_{3}{\times} C_{17}&32&32&1024&1355&1024&1355&\ref{Pr:HK}\\
\textbf{102}&&&&&\textbf{1355}&&\textbf{1355}&\\
\hline
103/1&C_{103}&102&102&10404&10505&10404&10505&\ref{Th:Cp}\\
\textbf{103}&&&&&\textbf{10505}&&\textbf{10505}&\\
\hline
104/2&C_{8}{\times} C_{13}&48&48&2304&2480&2304&2480&\ref{Pr:HK}\\
104/9&C_{4}{\times} C_{2}{\times} C_{13}&96&60&4032&4340&3168&3410&\ref{Pr:HK}\\
104/14&C_{2}^{3}{\times} C_{13}&2016&72&28368&52855&4608&5425&\ref{Pr:HK}\\
\textbf{104}&&&&&\textbf{59675}&&\textbf{11315}&\\
\hline
105/2&C_{3}{\times} C_{5}{\times} C_{7}&48&48&2304&3895&2304&3895&\ref{Pr:HK}\\
\textbf{105}&&&&&\textbf{3895}&&\textbf{3895}&\\
\hline
106/2&C_{2}{\times} C_{53}&52&52&2704&2755&2704&2755&\ref{Pr:HK}\\
\textbf{106}&&&&&\textbf{2755}&&\textbf{2755}&\\
\hline
107/1&C_{107}&106&106&11236&11341&11236&11341&\ref{Th:Cp}\\
\textbf{107}&&&&&\textbf{11341}&&\textbf{11341}&\\
\hline
108/2&C_{4}{\times} C_{27}&36&36&1296&1764&1296&1764&\ref{Pr:HK}\\
108/5&C_{2}^{2}{\times} C_{27}&108&54&3564&6615&2592&3969&\ref{Pr:HK}\\
108/12&C_{4}{\times} C_{9}{\times} C_{3}&216&40&3456&5424&1344&2112&\ref{Pr:HK}\\
108/29&C_{2}^{2}{\times} C_{9}{\times} C_{3}&648&60&9504&20340&2688&4752&\ref{Pr:HK}\\
108/35&C_{4}{\times} C_{3}^{3}&22464&48&92944&137284&1936&2420&\ref{Pr:HK}\\
108/45&C_{2}^{2}{\times} C_{3}^{3}&67392&72&255596&514815&3872&5445&\ref{Pr:HK}\\
\textbf{108}&&&&&\textbf{686242}&&\textbf{20462}&\\
\hline
109/1&C_{109}&108&108&11664&11771&11664&11771&\ref{Th:Cp}\\
\textbf{109}&&&&&\textbf{11771}&&\textbf{11771}&\\
\hline
110/6&C_{2}{\times} C_{5}{\times} C_{11}&40&40&1600&2071&1600&2071&\ref{Pr:HK}\\
\textbf{110}&&&&&\textbf{2071}&&\textbf{2071}&\\
\hline
111/2&C_{3}{\times} C_{37}&72&72&5184&6655&5184&6655&\ref{Pr:HK}\\
\textbf{111}&&&&&\textbf{6655}&&\textbf{6655}&\\
\hline
\end{array}
\end{displaymath}
\newpage

\begin{displaymath}\
\begin{array}{|rr|rr|rr|rr|r|}
\hline
G & \text{structure} & |A| & |X| & |O| & cq & |O_c| & mq & \text{ref}\\
\hline\hline
112/2&C_{16}{\times} C_{7}&48&48&2304&2624&2304&2624&\ref{Pr:HK}\\
112/19&C_{4}^{2}{\times} C_{7}&576&84&14400&25584&6048&7708&\ref{Pr:HK}\\
112/22&C_{8}{\times} C_{2}{\times} C_{7}&96&60&4032&4592&3168&3608&\ref{Pr:HK}\\
112/37&C_{4}{\times} C_{2}^{2}{\times} C_{7}&1152&78&20304&33620&5256&6150&\ref{Pr:HK}\\
112/43&C_{2}^{4}{\times} C_{7}&120960&84&746892&1630447&5760&7339&\ref{Pr:HK}\\
\textbf{112}&&&&&\textbf{1696867}&&\textbf{27429}&\\
\hline
113/1&C_{113}&112&112&12544&12655&12544&12655&\ref{Th:Cp}\\
\textbf{113}&&&&&\textbf{12655}&&\textbf{12655}&\\
\hline
114/6&C_{2}{\times} C_{3}{\times} C_{19}&36&36&1296&1705&1296&1705&\ref{Pr:HK}\\
\textbf{114}&&&&&\textbf{1705}&&\textbf{1705}&\\
\hline
115/1&C_{5}{\times} C_{23}&88&88&7744&9595&7744&9595&\ref{Pr:HK}\\
\textbf{115}&&&&&\textbf{9595}&&\textbf{9595}&\\
\hline
116/2&C_{4}{\times} C_{29}&56&56&3136&3244&3136&3244&\ref{Pr:HK}\\
116/5&C_{2}^{2}{\times} C_{29}&168&84&8624&12165&6272&7299&\ref{Pr:HK}\\
\textbf{116}&&&&&\textbf{15409}&&\textbf{10543}&\\
\hline
117/2&C_{9}{\times} C_{13}&72&72&5184&7440&5184&7440&\ref{Pr:HK}\\
117/4&C_{3}^{2}{\times} C_{13}&576&96&19584&28365&8064&10540&\ref{Pr:HK}\\
\textbf{117}&&&&&\textbf{35805}&&\textbf{17980}&\\
\hline
118/2&C_{2}{\times} C_{59}&58&58&3364&3421&3364&3421&\ref{Pr:HK}\\
\textbf{118}&&&&&\textbf{3421}&&\textbf{3421}&\\
\hline
119/1&C_{7}{\times} C_{17}&96&96&9216&11111&9216&11111&\ref{Pr:HK}\\
\textbf{119}&&&&&\textbf{11111}&&\textbf{11111}&\\
\hline
120/4&C_{8}{\times} C_{3}{\times} C_{5}&32&32&1024&1520&1024&1520&\ref{Pr:HK}\\
120/31&C_{4}{\times} C_{2}{\times} C_{3}{\times} C_{5}&64&40&1792&2660&1408&2090&\ref{Pr:HK}\\
120/47&C_{2}^{3}{\times} C_{3}{\times} C_{5}&1344&48&12608&32395&2048&3325&\ref{Pr:HK}\\
\textbf{120}&&&&&\textbf{36575}&&\textbf{6935}&\\
\hline
121/1&C_{121}&110&110&12100&13288&12100&13288&\ref{Th:Cp}\\
121/2&C_{11}^{2}&13200&120&144200&158199&13400&14508&\\
\textbf{121}&&&&&\textbf{171487}&&\textbf{27796}&\\
\hline
122/2&C_{2}{\times} C_{61}&60&60&3600&3659&3600&3659&\ref{Pr:HK}\\
\textbf{122}&&&&&\textbf{3659}&&\textbf{3659}&\\
\hline
123/1&C_{3}{\times} C_{41}&80&80&6400&8195&6400&8195&\ref{Pr:HK}\\
\textbf{123}&&&&&\textbf{8195}&&\textbf{8195}&\\
\hline
124/2&C_{4}{\times} C_{31}&60&60&3600&3716&3600&3716&\ref{Pr:HK}\\
124/4&C_{2}^{2}{\times} C_{31}&180&90&9900&13935&7200&8361&\ref{Pr:HK}\\
\textbf{124}&&&&&\textbf{17651}&&\textbf{12077}&\\
\hline
125/1&C_{125}&100&100&10000&12325&10000&12325&\ref{Th:Cp}\\
125/2&C_{25}{\times} C_{5}&2000&104&47200&66580&9280&13270&\\
125/5&C_{5}^{3}&1488000&120&?&?&?&?&\\
\textbf{125}&&&&&\textbf{?}&&\textbf{?}&\\
\hline
126/6&C_{2}{\times} C_{9}{\times} C_{7}&36&36&1296&1968&1296&1968&\ref{Pr:HK}\\
126/16&C_{2}{\times} C_{3}^{2}{\times} C_{7}&288&48&4896&7503&2016&2788&\ref{Pr:HK}\\
\textbf{126}&&&&&\textbf{9471}&&\textbf{4756}&\\
\hline
127/1&C_{127}&126&126&15876&16001&15876&16001&\ref{Th:Cp}\\
\textbf{127}&&&&&\textbf{16001}&&\textbf{16001}&\\
\hline
\end{array}\end{displaymath}

\end{small}


\begin{thebibliography}{99}


\bibitem{Dra}
A. Dr\'apal, \emph{Group isotopes and a holomorphic action}, Result. Math. \textbf{54} (2009), no. \textbf{3}--\textbf{4}, 253--272.

\bibitem{GAP}
The GAP Group, GAP -- Groups, Algorithms, and Programming, Version 4.5.5; 2012. {\tt http://www.gap-system.org}

\bibitem{HR}
C. J. Hillar, D. L. Rhea, \emph{Automorphisms of finite abelian groups}, Amer. Math. Monthly \textbf{114} (2007), 917--923.

\bibitem{Hou}
X. Hou, \emph{Finite modules over $\Z[t,t^{-1}]$}, J. Knot Theory Ramifications \textbf{21} (2012), no. \textbf{8}, 1250079, 28 pp.


\bibitem{Kir}
O. U. Kirnasovsky, \emph{Linear isotopes of small order groups}, Quasigroups and Related Systems \textbf{2} (1995), no. \textbf{1}, 51--82.

\bibitem{Kir1}
O. U. Kirnasovsky, \emph{Binary and $n$-ary isotopes of groups: fundamental algebraic properties and characterizations}, PhD thesis, Kiev 2001 (Ukrainian).

\bibitem{LOOPS}
G. P.~Nagy and P.~Vojt\v{e}chovsk\'y, \texttt{LOOPS}: Computing with quasigroups and loops in GAP, version 3.0.0, {\tt www.math.du.edu/loops}.

\bibitem{OEIS}
OEIS Foundation Inc. (2011), The On-Line Encyclopedia of Integer Sequences, {\tt http://oeis.org}.


\bibitem{Smi}
J.D.H. Smith, \emph{Finite equationally complete entropic quasigroups}, Contributions to general algebra (Proc. Klagenfurt Conf. 1978), 345--356 (1979).

\bibitem{Smi-book}
J.D.H. Smith, \emph{An introduction to quasigroups and their representations}, Chapman \& Hall/CRC, 2007.

\bibitem{Sok}
F. Sokhatsky, \emph{On isotopes of groups, I, II, III}, Ukrainian Math. J. 47/10 (1995), 1585--1598; 47/12 (1995), 1935--1948; 48/2 (1996), 283--293.

\bibitem{SS}
F. Sokhatsky, P. Syvakivskij, \emph{On linear isotopes of cyclic groups}, Quasigroups and Related Systems \textbf{1} (1994), no. \textbf{1}, 66--76.

\bibitem{Sta-latin}
D. Stanovsk\'y, \emph{A guide to self-distributive quasigroups, or latin quandles}, Quasigroups and Related Systems \textbf{23} (2015), no. \textbf{1}, 91--128.

\bibitem{SSta}
M. Stronkowski, D. Stanovsk\'y, \emph{Embedding general algebras into modules}, Proc. Amer. Math. Soc. \textbf{138} (2010), no. \textbf{8}, 2687--2699.

\bibitem{Sz}
\'A. Szendrei, \emph{Modules in general algebra}, Contributions to general algebra \textbf{10} (Proc. Klagenfurt Conf. 1997), 41--53 (1998).


\end{thebibliography}
\end{document}